\documentclass[12pt]{article}
\usepackage{enumerate}
\usepackage[hyperindex,breaklinks,colorlinks=true,linkcolor=black,anchorcolor=black,citecolor=black,filecolor=black,menucolor=black,runcolor=black,urlcolor=black]{hyperref}
\usepackage[latin1]{inputenc}
\usepackage{amsmath, amsthm, amsfonts, amssymb}
\usepackage{mathrsfs}
\usepackage{latexsym}
\usepackage{microtype}
\usepackage{fullpage}
\usepackage[all]{xy}
\usepackage{color}
    \usepackage{units}
        \usepackage{xspace}

	\usepackage{xcolor}
	\usepackage{soul}
	
	\definecolor{lightblue}{rgb}{.60,.60,1}
	\definecolor{lightred}{rgb}{1, .60,.60}

\newtheorem{theorem}{Theorem}
\newtheorem{lemma}[theorem]{Lemma}
\newtheorem{proposition}[theorem]{Proposition}

\newtheorem{question}[theorem]{Question}

\theoremstyle{definition}
\newtheorem{definition}[theorem]{Definition}

\theoremstyle{remark}

\newtheorem{rmks}{Remarks}[section]

\newtheorem*{Examples}{Examples}

\newcommand{\A}{\mathcal{A}}

\newcommand{\FF}{\mathbb{F}}

\newcommand{\N}{\mathbb{N}}

\newcommand{\R}{\mathbb{R}}

\newcommand{\dom}{\text{dom}}
\newcommand{\po}{\mathbb{P}}

\definecolor{purple}{rgb}{.9,0.2,.9}

\newcommand{\cs}{2^\omega}
\newcommand{\uh}{{\upharpoonright}}
\renewcommand{\phi}{\varphi}

\usepackage{color}

\title{From Bi-immunity to Absolute Undecidability}
\author{Laurent Bienvenu\thanks{This work was done as part of the France-Berkeley project \emph{Algorithmic randomness for non-uniform measures}.}, Adam R. Day\thanks{Day was supported by a Miller Research Fellowship in the Department of Mathematics at the University of California, Berkeley.}, 
Rupert H\"olzl\thanks{H\"olzl was supported by a Feodor Lynen postdoctoral research fellowship by the
Alexander von Humboldt Foundation.}}

\date{\today} 
\begin{document}


\maketitle


\begin{abstract}
An infinite binary sequence~$A$ is absolutely undecidable if it is impossible to compute~$A$ on a set of positions of positive upper density. Absolute undecidability is a weakening of bi-immunity.  Downey, Jockusch and Schupp~\cite{DowneyJockuschSchupp} asked whether, unlike the case for bi-immunity,  there is an absolutely undecidable set in every non-zero Turing degree. We provide a positive answer to this question by applying techniques from coding theory. We show how to use Walsh-Hadamard codes to build a truth-table functional which maps any sequence $A$ to a sequence $B$, such that given any restriction of $B$ to a set of positive upper density, one can recover $A$. This implies that if $A$ is non-computable, then $B$ is absolutely undecidable.
 Using a forcing construction, we show that this result cannot be strengthened in any significant fashion. 
\end{abstract}

\section{Introduction}
Let $A \in \cs$ be a non-computable infinite binary sequence. How can we gauge the extent to which $A$ is close to being computable? 
Many basic concepts in computability theory  can be seen as providing a partial answer to this question. For example, if $A$ is low, $\Delta^0_2$, or of minimal degree, then $A$ could be seen as being close to computable.

In this paper, we will consider a different perspective. Given~$A$ and a partial computable function $\phi$ we will say $A$ \textit{extends} $\phi$ if for all $x \in \dom(\phi)$, $A(x) = \phi(x)$ (where $A(n)$ is the value of the $n$-th bit of~$A$). We will measure how close $A$ is to being computable by considering the size of the domain of any partial computable function that $A$ extends (the larger the domain, the closer $A$ is to being computable). 

A set  $A$ is bi-immune if neither $A$ nor its complement contain an infinite c.e.\ (or, equivalently, computable) subset. It is not difficult to see that $A$ is bi-immune if and only if $A$ does not extend a partial computable function with infinite domain. Hence, from this paper's perspective,  bi-immune sets are as far from computable as possible. 
Once bi-immune sets are shown to exist, the property of being bi-immune is clearly not invariant under Turing equivalence. 
However, it does make sense to ask whether every non-zero Turing degree contains a bi-immune set.  The answer is~no.

\begin{theorem}[Jockusch \cite{Jockusch}]\label{thm:jockush}
There exists a non-zero  Turing degree such that no element of this degree is bi-immune.
\end{theorem}

To determine whether a sequence~$A$ is bi-immune, we asked whether $A$ extends a partial computable function with infinite domain. Essentially, bi-immunity  regards the domain of a partial computable function as  large if it is infinite. We can weaken the notion of bi-immunity by strengthening our concept of largeness.

\begin{definition}
\label{def:density}
Let $D \subseteq \N$. The \emph{upper density of $D$} is the quantity
\[\rho(D):=\limsup_{n\rightarrow \infty} \frac{|D \cap \{0, \ldots n-1\}|}{n}  .\]
When the upper density of $D$ is zero, we simply say that $D$ has \emph{density 0}.\footnote{Density,  defined as above but using $\lim$ instead of $\limsup$, is often denoted by $\rho$. As density  will not play a significant role in this paper we will simplify notation by using $\rho$ to denote upper density instead of the more common $\overline{\rho}$.}
\end{definition}


\begin{definition}[Myasnikov-Rybalov \cite{MyasnikovRybalov}] We say that $A \in \cs$ is \textit{absolutely undecidable} if there is no partial computable function $\phi: \N \rightarrow \{0,1\}$ whose domain has positive upper density and such that $\phi(n)=A(n)$ for all~$n$ in the domain of~$\phi$. 
\end{definition}





The notion of absolute undecidability comes from  the study of generic computability~\cite{DowneyJockuschSchupp, JockuschSchupp, KapovichEtAl}. A sequence $A \in \cs$ is generically computable if it can be ``computed modulo a set of density $0$", i.e., if there is a partial computable function $\varphi$ such that $\phi(n)=A(n)$ for all~$n$ in the domain of $\phi$, and $\rho(\N \setminus \dom(\phi))=0$. Absolute undecidability is the antithesis of generic computability; a set is absolutely undecidable if it cannot be computed at all, modulo a set of positions of density $0$.

In Downey, Jockusch and Schupp~\cite{DowneyJockuschSchupp}, the following question is asked. 

\begin{question}
\label{quest:jockusch}
Does every non-zero Turing degree contain a set which is absolutely undecidable?
\end{question}

In other words, does Theorem~\ref{thm:jockush} fail if we weaken bi-immunity to absolute undecidability? We give a positive answer to this question in Theorem~\ref{thm:abs-incomp-coding}. In fact every non-zero \emph{truth-table} degree contains a sequence which is absolutely undecidable. 

Theorems~\ref{thm:jockush} and \ref{thm:abs-incomp-coding} show that an interesting dichotomy exists between bi-immunity and absolute undecidability. In Section~\ref{sect:between}, of this paper we investigate what happens between these two notions. 

To express this idea precisely, we need to introduce some definitions. 
For a function $h\colon\N\rightarrow\R^{>0}$, denote by $o(h)$ the set of functions $f\colon\N \rightarrow \R$ such that $\lim_{n\rightarrow \infty} {f(n)}/{h(n)} = 0$. In particular, write $o(n)$ for $o(g)$ when $g$ is the function $n \mapsto n$.
Given $D \subseteq \N$, define
$\kappa_D\colon n \mapsto |D \cap \{0, \ldots,n-1\}|$ and $\rho_D\colon n \mapsto \kappa_D(n)/n$.
 Theorem~\ref{thm:abs-incomp-coding} is a uniform version of the following statement.
\begin{quote}
Every non-zero truth-table degree contains an element $X$ such that if $\phi$ is a partial computable function and $X$ extends $\phi$ then $\kappa_{\dom(\phi)} \in o(n)$.
\end{quote}

To determine what happens between bi-immunity and absolute undecidability, we can ask whether this statement can be strengthened by replacing  $o(n)$ with a smaller set of functions. In particular, is this statement still true if we replace $o(n)$ with $o(h)$ for some computable  function $h \in o(n)$ e.g.\ $n/\log\log n$?

We will show that the statement cannot be strengthened in this manner.  If $h$ is a computable function and $h(n) \in o(n)$, then there is a non-zero Turing degree $\mathbf{a}$ such that for all sets $X  \le_T \mathbf{a}$, $X$ extends a partial computable function $g$ with $\kappa_{\dom(g)} \not \in o(h)$.   In fact we will construct a single Turing degree that works for all such $h$. This is the content of Theorem~\ref{thm: tight}.
Hence the threshold where the behavior changes between bi-immunity and absolute undecidability occurs right at the absolute undecidability end. 

Theorem~\ref{thm: tight} is proved using forcing with computable perfect trees. The structure of this proof is similar to that of Jockusch's proof of Theorem~\ref{thm:jockush}~\cite{Jockusch}, but requires some extra machinery to deal with density.

\section{Coding Theory and Absolute Undecidability}
Our objective in this section and the following section,  is to answer Question~\ref{quest:jockusch} by proving the following theorem.

\begin{theorem} \label{thm:abs-incomp-coding}
There exists a tt-functional $\Phi$ such that for every non-computable sequence $A \in \cs$, $\Phi^A$ is absolutely undecidable and $\Phi^A \equiv_{tt} A$.
\end{theorem}

Let us present the structure of our proof.  We will show how to encode any sequence~$A$ via a tt-reduction $\Phi$ into a sequence~$B$ in such a way that $A$ can be fully recovered (by another tt-reduction) given $B$ on a set of positive upper density. If $A$ is non-computable, then $B=\Phi^A$ is necessarily absolutely undecidable. Otherwise, $B$ would extend a partial computable function $\varphi$ with domain a set of positive upper density and thus $A$ could be computed using~$\varphi$.

Our proof will have an element of non-uniformity. Instead of  reproducing the set $A$ exactly (given $B$ on a set of positive upper density)  we will  produce a computable tree of bounded width having~$A$ as a path. This is sufficient because any path through a computable tree of bounded width is itself computable. Theorem \ref{thm_nonuni_neces} shows that this non-uniformity cannot be avoided, and, in fact, that  even any finite number of decoding procedures do not suffice to decode an arbitrary set $A$. This is shown to hold for an arbitrary coding scheme, not just for the specific scheme $\Phi$ used in the proof of Theorem \ref{thm:abs-incomp-coding}.

The functional $\Phi$ is built upon an error correcting code known as a Walsh-Hadamard code (see, e.g.\ Arora and Barak \cite[Section 19.2.2]{AroraBarak}). An error correcting code maps finite strings to finite strings, thus for our purposes we will have to use the Walsh-Hadamard code iteratively. On input $A$, for all~$n$, the initial segment (or prefix) of $A$ of length~$n$, which we denote by $A \uh n$, is encoded into a string $\sigma_n$. The image of~$A$ under~$\Phi$ will be the sequence $B=\sigma_1\sigma_2\sigma_3\ldots.$

Ideally we would like a computable sequence of functions $\{f_n\}_{n \in \N}$, where $f_n$ is a function from binary strings of length $n$ to binary strings of some length $k_n$ with the following property. For any $\delta >0$, for all sufficiently large $n$, for all strings $\sigma$ of length $n$,  any $\lfloor \delta \cdot k_n\rfloor$ bits of $f(\sigma)$  uniquely determines $\sigma$. In coding theory, such an coding is called an \emph{erasure code}: one encodes a string $\sigma$ into $\tau=f(\sigma)$ in such a way that $\sigma$ can be recovered from any subset $\tau'$ of the bits of $\tau$ of sufficiently large size (think of~$\tau'$ as a version of~$\tau$ where some of the bits are missing, say replaced with `?', but those that are given are correct).  

%
 If there were such a coding, we could attempt to prove Theorem~\ref{thm:abs-incomp-coding} by defining $B$ to be equal to 
\[f_1(A\uh1)f_2(A\uh2)f_4(A\uh4)f_8(A\uh8)\ldots.\]
Now if the sequence $k_n$ has the property that given $B$ on a set of positive upper density $\delta$, then for some constant $c >0$,  for infinitely many $n$, one must have $\lfloor \delta/c \cdot k_n\rfloor$ bits of $f_n(A\uh n)$, then we could  reconstruct $A\uh n$ for these $n$, and hence recover~$A$ itself.
While not all sequences $k_n$ have this property, the sequence defined by letting $k_n =2^n$ does (as we will prove in Lemma~\ref{lem:density}). This is the sequence we will make use of in our proofs.

Unfortunately, requiring that any $\lfloor \delta \cdot k_n\rfloor$ bits of $f(\sigma)$  uniquely determine $\sigma$ as described above, is to much to ask for.\label{pg:list-dec-example}
There is a theoretical limit to recovering an encoded string unambiguously from a small fraction of the bits of its encoding.\footnote{The same problem occurs for any finite alphabet size.} For simplicity let us say that we only want to encode three different strings $\sigma_1, \sigma_2, \sigma_3$ into codewords $\tau_1, \tau_2, \tau_3$. Without loss of generality we can assume that the codewords all have the same length~$m$. For each position $i \leq m$, at least two codewords agree on the bit number $i$. Thus, by the pigeonhole principle, there exist two strings among $\tau_1,\tau_2,\tau_3$ that agree on at least one third of their bits. 
%

This seems to defeat the plan of retrieving the original string from a small fraction of the bits of its encoding, as the simple example above shows that any fraction smaller that $1/3$ is insufficient in general. The solution to this problem is  to use another concept from coding theory: \emph{list decoding}. In the remainder of this section we will review these concepts and explain how they can be used to prove Theorem~\ref{thm:abs-incomp-coding}. However, we appreciate that some readers will not have a background in coding theory, so in the following section we will give a proof of Theorem~\ref{thm:abs-incomp-coding} using basic facts about vector spaces, that assumes no prior knowledge of coding theory.

A \emph{list decodable code} is an error correcting coding scheme such that, if the codeword $z$ for a string~$x$ is known without too many errors, one can produce a small list of strings to which~$x$ belongs. Define the (normalized) Hamming distance between two words of equal length $n$ as $d_H(x,y):=|\{i\mid x(i) \not= y(i)\}|/n$. The \emph{minimum distance} of a coding scheme is the minimum  normalized Hamming distance between two distinct codewords. List decodable codes  ensure that a ball (in Hamming distance) around a potential codeword contains only a small number of actual codewords. This means that a potential codeword can be decoded into a small number of possible original strings. In our case this number will be constant, which will allow us to build a computable tree of bounded width.

The Walsh-Hadamard code, which we define in the next section, maps words of length $n$ to codewords of length $2^n$. In~\cite[Section 19.2.2]{AroraBarak} it is shown that this is a code with minimum distance $1/2$, i.e., any two codewords $x,y$ 
 disagree on at least half of their bits. Assume we have a ``corrupted" version of a codeword that could either be $x$ or $y$, where corrupted means that some bits $b$ have been replaced by $1-b$. Assume that we are guaranteed that the corruption has only happened on at most a quarter of all bits.
Then we can be sure that fewer than half of the bits where $x$ and $y$ disagree have been corrupted, so we can take a majority vote on those bits to decide whether the given corrupted codeword came from~$x$ or~$y$.
 This is completely insufficient for our purposes, of course, since in our construction we will in general only be guaranteed a much smaller fraction $\varepsilon \ll 1/4$ of bits of the codeword. But we can apply the following theorem:

\begin{theorem}[Johnson bound (see, e.g., \cite{AroraBarak})]
If $E: \{0,1\}^n \rightarrow \{0,1\}^m$ is an error correcting code of minimum distance at least $1/2$, then for every $x \in \{0,1\}^m$ and $\delta > 0$, there exist at most $l=1/(2\delta^2)$ elements $y$ of $\{0,1\}^n$ such that $d_H(x,E(y)) \leq 1/2-\delta$.
\end{theorem}

In other words, if we have a corrupt version $z'$ of a codeword $z=E(y)$, which is good enough in the sense that a fraction $1/2 + \delta$ of the bits of $z'$ coincide with those of~$z$, we can generate from $z'$ a list of size $1/(2\delta^2)$ of possible candidates for~$y$ which will indeed contain~$y$.


In our case, we do not have such a $z'$, but we have instead a known fraction of bits of $z$ which we are sure are correct and we know their position in~$z$. As we explained above, what we need for this situation is an erasure code. 


The following argument shows that we can still use the Walsh-Hadamard code in this context. Imagine we know a fraction $2\delta$ of the bits of a codeword $z$ for the original word~$y$, together with their respective positions in $z$. Now fill all the remaining positions with $0$'s to get a potential codeword $z_0$ and with $1$'s to get a potential codeword~$z_1$. Perform a list-decoding on both $z_0$ and $z_1$, thus getting two lists of possible candidates of size at most $1/(2\delta^2)$. Merge the two lists; the resulting list has size at most $2 \cdot 1/(2\delta^2) = 1/\delta^2$. Note that the correct entry is on the list since either $z_0$ or $z_1$ must have a fraction at least $1/2+\delta$ of its bits in common with~$z$, so the list contains the original word~$y$.

\section{Every truth table degree contains an absolutely undecidable set}

We hope that the reader who is familiar with coding theory can find in the previous section sufficient information to prove Theorem~\ref{thm:abs-incomp-coding}. In this section, we present our argument in a self-contained way that assumes no prior knowledge of coding theory. Besides, the direct analysis of the Walsh-Hadamard code for our purposes will give a better bound on the list of candidates than the one we gave in the previous section ($1/\delta$ instead of $1/\delta^2$).

The core of the proof resides in a combinatorial argument about vector spaces. Before giving the argument let us illustrate it with an example. Define $f_3: \{0,1\}^3\rightarrow \{0,1\}^7$ as follows. If $a,b,c \in \{0,1\}$ then
\[f_3(abc) = abc(a+b)(b+c)(a+c)(a+b+c),\]
where addition is performed modulo $2$ (caveat: the right-hand side is a concatenation of bits, not a product modulo 2). In other words, $f_3(abc)$ is obtained by concatenating together the output from all linear functionals from $\{0,1\}^3$ to $\{0,1\}$ with argument $abc$. (Note that we do not include the zero functional in this example.) Given any $a,b,c \in \{0,1\}$ assume we are given 3 bits of $f_3(abc)$, e.g.\  the zeroth bit is $1$, the fourth bit is $0$ and sixth bit is~$1$. The zeroth bit tells us that $a=1$, the fourth bit tells us that $b+c=0$, the sixth bit is redundant information because we already know that $a+b+c=1$ from the zeroth and fourth bit. 
Nevertheless, we can determine that $abc \in \{100, 111\}$ i.e.\ we know that $abc$ is one of two possibilities. Further, no matter which $3$ bits we are given, we are always guaranteed at least two `independent' pieces of information and hence if we cannot determine $abc$, we can always show that there are only two possible values it could take. 

 What is remarkable is how this example scales. If we define $f_n$ in a similar manner for all~$n$, then for any string~$\sigma$ of length~$n$, given $1/4$ of the bits of $f_n(\sigma)$ we can find a set of size at most two that $\sigma$ must belong to. Crucially, the maximum size of the set does not depend on~$n$, but only the fraction of the bits we have access to, in this case $1/4$.

Let $V$ be a vector space of dimension~$n$ over a finite field $\FF$ of cardinality~$q$. Note that in this paper we are only concerned with the case that $q=2$, but we will present the argument in its general form. Denote by $V^*$ the dual space of $V$, i.e., the set of linear functionals from $V$ to $\FF$. In the example above, we were given the values of some bits of $f_n(a,b,c)$. Each bit told us two things, the position of the bit told us the linear functional used, and the value of the bit told us the value of the linear functional on the argument $abc$.  We can represent this information as a subset $S$ of pairs in $(\varphi, z) \in V^* \times \FF$. 
For any $x \in V$, let $C_x=\{(\phi,\phi(x)) \in V^* \times \FF \mid \phi \in V^*\}$.  The set $C_x$ simply records the values $x$ takes on each linear functional in $V^*$. Now if $S\subseteq C_x$, then for all $(\phi, z) \in S$,  we have that $\phi(x)=z$ i.e.\ the values that $x$ takes on the linear functionals in $V^*$ is consistent with the information provided by $S$.


\begin{proposition}\label{prop:checksums} 
Let $V$, $n$, $\FF$, $q$, $V^*$ and $C_x$ be defined as above.
For any finite set $S$ of pairs $(\phi,z) \in V^* \times \FF$, the set $\A_S=\{x \mid S \subseteq C_x\}$  has  cardinality at most $q^n/|S|$.
\end{proposition}

\begin{proof}
Fix a finite set $S$ of pairs $(\phi,z) \in V^* \times \FF$. If $\A_S = \emptyset$ then the result holds trivially so let us assume that $\A_s$ is not empty and fix some element $x_0 \in \A_S$.
If some $\psi$ appears in two different pairs $(\psi,z_1)$ and $(\psi,z_2)$ of $S$, then it is clear that $\A_S=\emptyset$; so we may assume that any $\phi$ appears at most once in the pairs of $S$. Call the domain of $S$ (written $\dom(S)$) those $\phi$ that are the first coordinate of some element of $S$ and let $H$ be the subspace of $V^*$ generated by $\dom(S)$. Since $|H| \geq |S|$ we have $\dim(H) =  \log_q(|H|)  \geq \log_q(|S|)$, where the equality is due to the linear independence of vectors in a basis.

If $y \in \A_S$, this means that for all $\phi \in \dom(S)$, $\phi(x_0)=\phi(y)$, and therefore by linearity, that $\phi(x_0)=\phi(y)$ for all $\phi \in H$, which again by linearity can be re-written as $\phi(x_0-y)=0$ for all $\phi \in H$. In other words, $x_0-y$ belongs to the annihilator $H^\circ$ of $H$. 
Of course, the converse holds, i.e.\ if $x_0-y$ belongs to $H^\circ$, then $\phi(x_0)=\phi(y)$ for all $\phi \in \dom(S)$ and therefore $y \in \A_S$. This shows that $\A_S$, provided it contains at least one element $x_0$, is the affine space $x_0 + H^\circ$ and therefore has the same dimension as $H^\circ$. Now, using the classical expression of the dimension of the annihilator, $\dim(\A_S)=\dim(H^\circ)=\dim(V^*)-\dim(H) = n - \dim(H) \leq n - \log_q(|S|)$. This implies that $|\A_S | \le q^n/|S|$.
\end{proof}

We will also need the following two  easy lemmas about density.

\begin{lemma} \label{lem:density}
For all integers~$n$ let $I_n=\{2^n,\ldots,2^{n+1}-1\}$ (note that the $I_n$ form a partition of $\N \setminus\{0\}$). If a set $D \subseteq \N$ has positive upper density greater than $\delta$, then for infinitely many~$n$, the density of $D$ inside $I_n$ (that is, the quantity $|D \cap I_n|/|I_n|$) is at least $\delta/2$.
\end{lemma}

\begin{proof}
Suppose for the sake of contradiction that $|D \cap I_n|/|I_n| < \delta/2$ for almost all~$n$. Without loss of generality, we may assume that this even holds for all~$n$ (by removing finitely many elements from~$D$, which does not change the upper density). For any given~$k$, let $n=n(k)$ be such that $k \in I_n=\{2^n,\ldots,2^{n+1}-1\}$. Then
\[
|D \cap \{0,\ldots,k\}| \leq  \sum_{j \leq n} |D \cap I_j| \leq \sum_{j \leq n} (\delta/2) \cdot 2^j \leq (\delta/2) \cdot 2^{n+1}  \leq (\delta/2) \cdot 2k \leq \delta k.
\]
This contradicts the fact that the upper density of~$D$ is greater than~$\delta$.
\end{proof}

%
\begin{lemma}
%
\label{lem: removal}
If $D \subseteq \N$ has upper density $\delta$ and $E \subseteq \N$ has density $0$ then
$D \setminus E$ has upper density $\delta$.
\end{lemma}

\begin{proof}
For all $n$, $|(D \setminus E) \cap \{0,\ldots,n-1\}| \geq |D \cap \{0,\ldots,n-1\}| - |E \cap \{0,\ldots,n-1\}| \geq |D \cap \{0,\ldots,n-1\}| - o(n)$. Thus $\limsup_n |(D \setminus E) \cap \{0,\ldots,n-1\}|/n = \limsup_n |D \cap \{0,\ldots,n-1\}|/n$
\end{proof}

We are now ready to prove Theorem~\ref{thm:abs-incomp-coding}.

\begin{proof}[Proof of Theorem~\ref{thm:abs-incomp-coding}]
We construct $\Phi$ block by block. On input $A$, for all~$n$, the initial segment $A \uh n$ is mapped to a string $\sigma_n$ of length $2^n$, and the image of $A$ under $\Phi$ is the sequence $B=\sigma_1\sigma_2\sigma_3\ldots.$ For all~$n$, $\sigma_n$ is constructed as follows:
\begin{itemize}
\item Identify $\{0,1\}^n$ with the vector space $V_n$ of dimension~$n$ over $\FF_2$, and let $x$ be the element of $V_n$ corresponding to $A \uh n$.
\item Order the elements of $V_n^*$ in some canonical way (say lexicographically, as $V_n^*$ can also be identified with $\{0,1\}^n$):  $V_n^*=\{\phi_1,\phi_2,\ldots,\phi_{2^n}\}$.
\item Define $\sigma_n$ to be the string $\phi_1(x)\phi_2(x)\ldots\phi_{2^n}(x)$.
\end{itemize}

Let us show that if $B=\sigma_1\sigma_2\sigma_3\ldots$  is not absolutely undecidable, then $A$ is computable. Assuming $B$ is not absolutely undecidable, let $f$ be a partial computable function whose domain $D$ has upper density greater than some rational $\delta>0$ and such that $f(k)=B(k)$ for all~$k \in D$. By Lemma~\ref{lem:density},  there are infinitely many~$n$ such that the density of $D$ in $I_n$ is at least $\delta/2$. Since $D$ is c.e.\ and $\delta$ rational, the set of such~$n$ is c.e.\ and therefore contains a computable set $\{n_0<n_1<n_2<n_3<\ldots\}$. Note that $B \uh I_n = \sigma_n$, so by Lemma \ref{lem:density}, for all~ $n_i$, we can compute a fraction $\delta/2$ of the values of the bits of $\sigma_{n_i}$ (taking the values of $f$ on elements in $I_{n_i}$). By the construction of $\sigma_{n_i}$, this means that we can compute a subset of $\{(\phi,\phi(x)) \mid \phi \in V^*_{n_i}\}$ (where $x$ is the element of $V_{n_i}$ corresponding to $A \uh {n_i}$) of size at least $(\delta/2)\cdot 2^{n_i}$, which, by Proposition~\ref{prop:checksums}, allows us to compute a finite set $\A_i$ of size at most $\frac{2^{n_i}}{(\delta/2)2^{n_i}}=2/\delta$ containing $A \uh n_i$. Thus, $A$ belongs to the $\Pi^0_1$ class
\[
\{X \in \cs \mid \forall i \; (X \uh n_i) \in \A_i\}.
\]
This $\Pi^0_1$ class contains at most $2/\delta$ elements as all $\A_i$ have size at most $2/\delta$. Thus all elements of this $\Pi^0_1$ class are computable, and therefore $A$ is computable.

It is clear that $A\leq_{tt}\Phi(A)$, as we only need the first $2^{n+1}-1$ bits of $\Phi(A)$ to recover~$A\uh n$.
\end{proof}

The proof of Theorem~\ref{thm:abs-incomp-coding} constructs a single functional $\Phi$ that uniformly encodes any sequence $A \in 2^{\omega}$ to $\Phi^A$. On the other hand, the reader might notice that the decoding procedure given for recovering $A$ given $\Phi^A$ on a set of positive upper density is not uniform. Indeed, finding a path in a computable tree which only has finitely many paths cannot be done uniformly in general. And even with a known bound on the number of paths, one cannot effectively compute a finite list of sequences which contains all paths. We shall formally prove that this cannot be avoided, not only for the Walsh-Hadamard coding we use, but for any other coding scheme as well. By analogy to the terminology used above, the next theorem can be informally stated as follows: infinitary erasure codes with finite list decoding do not exist. 

In this section we will freely identify $\cs$ with the powerset of $\N$. Fix a $tt$-functional $\Gamma$. A Turing functional $\Psi$ correctly decodes $X$ on $D$, if $\Psi$ can compute $X$ given access to the bits of $\Gamma^X$ in $D$, along with $D$ itself, i.e.\ if $X = \Psi(D \oplus (\Gamma^X \cap D))$. The following theorem shows that for a fixed $\delta$, there does not exist a $tt$-functional $\Gamma$, together with a finite number of functionals $\Psi_1,\ldots,\Psi_k$ such that for any~$X$, and any $D$ such that $\rho(D) \ge \delta$,  $X$ belongs to the set $\{\Psi_i(D\oplus(\Gamma^X \cap D)) \mid i \leq k\}$.

\begin{theorem}\label{thm_nonuni_neces}
Fix $k$. Let $\Gamma$ be a  $tt$-functional, 
 $\Psi_1, \Psi_2, \ldots, \Psi_k$ a list of Turing functionals and $\sigma$, $\tau$ finite strings. There exist computable $X$ and $D$ with $\sigma \preceq X$ and $\tau \preceq D$ such that:
 \begin{enumerate}[(i)]
 \item $\rho(D) \ge 1/3$.
 \item $X \not \in \{\Psi_i(D\oplus(\Gamma^X \cap D)) \mid i \leq k\}$.
 \end{enumerate}
\end{theorem}
\begin{proof}
The proof proceeds by induction on $k$. For $k=0$, there is nothing to prove. 
For $k+1$, let $\Gamma$ be a  $tt$-functional, 
 $\Psi_1, \Psi_2, \ldots, \Psi_{k+1}$ a list of Turing functionals and $\sigma$, $\tau$ finite strings.
 Find a finite string $\sigma' \succeq \sigma$ such that $|\Gamma^{\sigma'}| \ge |\tau|$. 
 Let $X_1,X_2,X_3$ be three distinct computable infinite binary sequences extending $\sigma'$ and let $Y_i=\Gamma^{X_i}$ for~$i \leq 3$. By the discussion of page~\pageref{pg:list-dec-example}, for all~$n$, there are $1 \leq i < j \leq 3$ such that $Y_i \uh n$ and $Y_j \uh n$ coincide on at least one third of their bits. Thus there are $1 \leq i < j \leq 3$ such that $Y_i \uh n$ and $Y_j \uh n$ coincide on at least one third of their bits for infinitely many~$n$. Without loss of generality, assume this holds for the pair $(Y_1,Y_2)$. The set $D$ of positions~$n$ such that $Y_1(n)=Y_2(n)$ has upper density at least~$1/3$. Further as $Y_1$ and $Y_2$ must agree on the first $|\tau|$ bits,  we can adjust $D$ so that $\tau \preceq D$ without affecting the upper density of $D$ or the fact that for all $n \in D$, $Y_1(n)=Y_2(n)$. 
 
If $X_1 \not \in  \{\Psi_i(D\oplus(Y_1 \cap D)) \mid i \leq k+1\}$, then $X_1$ and $D$ witness that the theorem holds for $\Gamma$,  $\Psi_1, \Psi_2, \ldots, \Psi_{k+1}$,  $\sigma$ and $\tau$. 
Otherwise, we can assume without loss of generality that $X_1 = \Psi_{k+1}(D\oplus(Y_1 \cap D))$. Observe that there is some $m$ such that $X_1(m) \ne X_2(m)$. Further as $Y_1 \cap D = Y_2 \cap D$, and 
$\Psi_{k+1}(D\oplus(Y_2 \cap D)) = X_1$, there exists 
$\widehat{\tau} \preceq D$ and $\widehat{\sigma} \preceq X_2$ such that $|\widehat{\sigma}| \ge m+1$, and 
\[\Psi_{k+1}(\widehat{\tau} \oplus (\Gamma^{\widehat{\sigma}} \cap \widehat{\tau})) \succeq X_1 \uh (m+1).\]
By our induction hypothesis for $\Gamma$,   $\Psi_1, \Psi_2, \ldots, \Psi_k$, $\widehat{\sigma}$ and $\widehat{\tau}$, there is $\widehat{X}$, $\widehat{D}$ such that  $\widehat{\sigma} \preceq \widehat{X}$,    $\widehat{\tau} \preceq \widehat{D}$, and 
$\widehat{X} \not \in  \{\Psi_i(\widehat{D}\oplus(\Gamma^{\widehat{X}} \cap \widehat{D})) \mid i \leq k\}$. Observe that
$\Psi_{k+1} (\widehat{D}\oplus(\Gamma^{\widehat{X}} \cap \widehat{D})) \succeq X_1\uh (m+1) \ne \widehat{X} \uh (m+1)$ so 
in this case $\widehat{X}$ and $\widehat{D}$ witness that the theorem holds.
\end{proof}

\section{Between bi-immunity and absolute undecidability}
\label{sect:between}

Theorem~\ref{thm:abs-incomp-coding} shows that in every non-trivial tt-degree (and thus Turing degree) there is an absolutely undecidable set. In other words, in every such degree there is a set such that we are unable to infinitely often correctly guess a constant percentage of its bits. We show that the statement of the theorem is tight. Indeed, Theorem~\ref{thm: tight} below shows that there exists a sequence $X$ such that every $Y$ Turing below~$X$ is ``close" to not being absolutely undecidable.

 In the rest of this paper, $(\phi_e)_e$ and $(\Phi_e)_e$ denote respectively a standard enumeration of partial computable functions from $\N$ to $\N$ and a standard enumeration of partial Turing functionals from $\cs$ to $\cs$. We say that a Turing degree $d$ is \emph{computably dominated} if for any function $f \le_T d$, there is a computable function $g$ such that for all $n$, $g(n) \ge f(n)$. 


\begin{theorem}
\label{thm: tight}
There is a non-computable set $X$ such that for all $Y \le_T X$, and any computable function  $h \in o(n)$, there exists a partial computable function $\varphi$ such that:
\begin{enumerate}
\item $Y$ extends $\phi$ (i.e., $Y(n)=\phi(n)$ for all~$n \in \dom(\phi)$).
\item $\kappa_{\dom(\phi)} \not \in o(h)$.
\end{enumerate}
\end{theorem}

Theorem~\ref{thm: tight} is proved using forcing with computable perfect  trees. A tree $T \subseteq 2^{<\omega}$ is \emph{perfect} if for all $\sigma \in T$ there is some $\tau \succeq \sigma$ such that both
$\tau0$ and $\tau1$ are in $T$. Given a tree $T \subseteq 2^{<\omega}$, define $[T]$ to be the set of paths through $[T]$, i.e.\ $[T] = \{ X \in \cs \mid \forall n	 \, X\uh n \in T\}$. 

We define a partial order $\po$ on the set of  computable perfect trees in $2^{<\omega}$ by  $T_0 \le T_1$ if $[T_0] \subseteq [T_1]$.
For any computable $A \in \cs$, the set
\[D_A = \{ T \in \po \mid (\forall X \in [T])(X \ne A)\}\]
 is dense in $\po$. For every Turing functional $\Phi_e$, the set
\[ E_e=\{ T \in \po \mid ((\forall X \in [T])( \Phi^X_e \mbox{ is not total})) \vee ((\forall X \in [T])( \Phi^X_e \mbox{ is total}))\}\]
is also dense in $\po$. Hence if  $\{T_i\}_{i \in \N}$ is a descending sequence in $\po$  that meets all such dense sets, then any element of $\bigcap_{i \in \N} [T_i]$ is a non-computable set of computably dominated degree. Further such a sequence is computable by $\emptyset''$.

Since a Turing reduction to a set of computably dominated degree is always equivalent to a truth-table reduction it is enough to continue working with the latter class of reductions in the remainder of the proof.

In order to prove Theorem~\ref{thm: tight} we will show that if $h \in o(n)$ is a computable function and~$\Phi$ is a truth-table functional, then the set
\begin{equation}
\label{eq: dense}\{T \in \po \mid (\exists e)(\forall X \in [T])
(\Phi^X \mbox{ extends }\varphi_e \wedge \kappa_{\dom(\varphi_e)} \not \in o(h)\}
\end{equation}
is dense in $\po$.
%
%
%

%

Fix a computable perfect tree $T$, a truth-table functional $\Phi$ and a computable function~$h$ such that $h \in o(n)$. For all $\sigma \in T$ we will define two partial computable functions $f_\sigma : 2^{< \omega} \rightarrow T$ and $g_\sigma: \N \rightarrow \{0,1\}$ such that if $f_\sigma$ is total then the downward closure of the range of $f_\sigma$ is  a computable perfect subtree of $T$.  Further if $f_\sigma$ is total then  $g_\sigma$ will witness that the range of $f_\sigma$ is in \eqref{eq: dense}.

First let $f_\sigma(\lambda) = \sigma$ (where $\lambda$ is the empty string). Now assume that $f_\sigma$ has been defined on all strings of length less than or equal to $n$ and let
$\{\sigma_1, \ldots, \sigma_{2^n}\} = \{f_\sigma(\tau) \mid |\tau| = n\}$.
We search  the tree $T$ for some $m > n$, a set of nodes $\{\xi_1, \ldots, \xi_{2^n}\} \subseteq T$ and $D\subseteq \{0, \ldots, m-1\}$ such that:
\begin{enumerate}[(i)]
\item $\forall i \in \{1, \ldots, 2^n\}$, $\xi_i \succeq \sigma_i$.
\item $\forall x \in D$, $\Phi^{\xi_1}(x) = \Phi^{\xi_2}(x) = \ldots =  \Phi^{\xi_{2^n}}(x)$.
\item $|D| > h(m)$.
\end{enumerate}

If the search succeeds, then for all $x \in D$ we define $g_\sigma(x)$ to be the common value (i.e.\ $g_\sigma(x) = \Phi^{\xi_1}(x)$). We set
$f(\sigma_i 0)$ to be  one immediate successor of $\xi_i$ in $T$ and
$f(\sigma_i 1)$ to be the other immediate successor of $\xi_i$ in $T$.

We claim that there exists some $\sigma \in T$ such that the function $f_\sigma$ is total. To prove this, we will define a condition $\star$ such that if $\star$ holds for $\sigma$ then $f_\sigma$ is total. Further if $\star$ does not hold for any $\sigma$ then $f_\lambda$ is total (and though we will not need this, in fact $f_\sigma$ will be total for any $\sigma \in T$).

\begin{quote}
Condition~$\star$ holds for $\sigma  \in T$ if there exists $D \subseteq \N$ of positive upper density such that for all $X, Y \in [T] \cap [\sigma]$, the set $\{ x \in D \mid \Phi^X(x) \ne \Phi^Y(x)\}$ has density zero.
\end{quote}

As this condition is the central new idea in the proof of Theorem~\ref{thm: tight}, we will attempt to explain the underlying intuition.
Essentially, the functional $\Phi$ can behave in one of two ways  with respect to $[T] \cap [\sigma]$. Firstly, for some large set $D$, any two elements of $\Phi([T] \cap [\sigma])$   agree on ``almost all'' $n \in D$. Secondly, given any  large set $D$, there are two elements in $\Phi([T] \cap [\sigma])$ which disagree on a significant amount of $D$. Our approach for refining $T$ differs under these two cases. In the first case, that is, if  condition~$\star$ holds on $\sigma$, the refinement of $T$ is simple because no matter which paths we take in $[T] \cap [\sigma]$, their $\Phi$-images will always agree on ``almost all'' of $D$. This approach is pursued in Lemma~\ref{lem: star}. For the second case,  consider a situation in which we have a finite number of paths $\{X_0, \ldots, X_n\}$ in $[T]$ that agree  under $\Phi$ on a fixed large set $D$, that is,  if  $i, j \in \{0, \ldots, n\}$, $m \in D$ then $\Phi^{X_i}(m)=\Phi^{X_j}(m)$. 
If  we take two paths in $[T] \cap [\sigma]$ whose $\Phi$-images disagree on a significant amount $D$, then one of those paths must agree with paths in $\{X_0, \ldots, X_n\}$  on a significant amount of $D$ under $\Phi$. In Lemma~\ref{lem: not star} we will use this idea to build a suitable refinement of $T$.

\begin{lemma}
\label{lem: star}
If condition $\star$ holds for~$\sigma$ then $f_\sigma$ is total.
\end{lemma}
\begin{proof}
Let $D$ be a set of positive upper density that witnesses that condition $\star$ holds for $\sigma$. Assume that level $n$  of $f_\sigma$ has been defined and that this level is equal to $\{\sigma_1, \ldots, \sigma_{2^n}\}$. For all $i\in \{1, \ldots, 2^n\}$, let $X_i$ be the left-most path of $T$ above $\sigma_i$.

For $i\in  \{2, 3, \ldots, 2^n\}$  define $S_i = \{x \in D \mid \Phi^{X_1}(x) \ne \Phi^{X_i}(x)\}$. According to condition $\star$, for each $i$, the set $S_i$ has density $0$.
Thus if we let  $\hat{D} = D \setminus (S_2 \cup S_3 \cup \ldots \cup S_{2^n})$ we have that $\hat{D}$ is a set of positive upper density by repeated application of Lemma~\ref{lem: removal}. Further
for all $x \in \hat{D}$,
\[\Phi^{X_1}(x) = \Phi^{X_2}(x) = \ldots = \Phi^{X_{2^n}}(x).\]
As $h(n) \in o(n)$,  there exists some $m>n$ such that
$\rho_{\hat{D}}(m) > h(m)/m$.
For each $i$, let $\xi_i \in T$ be a sufficiently long initial segment of $X_i$  such that $\Phi^{\xi_i}(x)\downarrow$ for all $x \le m$. This means that $m$, $\{\xi_1, \xi_2, \ldots, \xi_{2^n}\}$ and $\hat{D} \cap \{0, \ldots, m-1\}$ meet the  conditions to define the next level of $f_\sigma$ and hence the search must end successfully at some point.
\end{proof}

\begin{lemma}
\label{lem: not star}
If condition $\star$ does not holds for any $\sigma \in T$ then $f_\lambda$ is total.
\end{lemma}
\begin{proof}
Assume that the $n$-th level of $f_\lambda$ has been defined and that this level is equal to $\{\sigma_1, \ldots, \sigma_{2^n}\}$. We argue by induction; let $D_1= \N$ and let $X_1$ be the left-most path above~$\sigma_1$.

For the induction step assume that for some $i < 2^n$, we have defined $D_i$ and $X_1 \succ \sigma_1, \ldots, X_i \succ \sigma_i$ such that $D_i$ is a set of positive upper density and for all $x \in D_i$
\[\Phi^{X_1}(x) = \Phi^{X_2}(x) = \ldots = \Phi^{X_{i}}(x).\]
The set $D_{i}$ has positive upper density; and since condition $\star$ fails above all $\sigma \in T$, it fails in particular above $\sigma_{i+1}$. Therefore there exist   $Y, Z \in [T] \cap [\sigma_{i+1}]$ such that the set
$E =\{x \in D_i \mid \Phi^Y(x) \ne \Phi^Z(x)\}$ does not have density zero i.e.\ $E$ has  positive upper density.  Now we can partition $E$ into the sets:
$E_1 = \{x \in E \mid \Phi^Y(x) = \Phi^{X_1}(x)\}$ and $E_2= \{x \in E \mid \Phi^Z(x) = \Phi^{X_1}(x)\}$. At least one of these sets has positive upper density  by Lemma~\ref{lem: removal}. If  $E_1$  has positive upper density we let $D_{i+1} = E_1$ and $X_{i+1} = Y$. Otherwise we let $D_{i+1} = E_2$ and $X_{i+1} = Z$.

Hence the set $D_{2^n}$ has positive upper density and for all $x\in D_{2^n}$ we have
\[\Phi^{X_1}(x) = \Phi^{X_2}(x) = \ldots = \Phi^{X_{2^n}}(x).\]
As argued in the previous lemma, this is sufficient to show that the construction of $f_\lambda$ can continue.
\end{proof}
\begin{proof}[Proof of Theorem \ref{thm: tight}]
We will show that if $\Phi$ is a truth-table functional and $h$ is a computable function such that
$h \in o(n)$, then the set \eqref{eq: dense} is dense in $\po$. Take any $T \in \po$. Consider the construction of the functions $f_\sigma$ and $g_\sigma$ with respect to $T$, $\Phi$ and $h$. By Lemmas~\ref{lem: star} and~\ref{lem: not star} for some $\sigma \in T$ the function $f_\sigma$ is total. Let $\hat{T}$ be the downward closure of the range of $f_\sigma$. $\hat{T}$ is a computable perfect tree and
$\hat{T} \le T$. Now for all $X \in [\hat{T}]$, $\Phi^X$  extends $g_\sigma$. Further $\kappa_{\dom({g_\sigma})} \not \in o(h)$ because when $f_\sigma$ is defined on all strings of length $n$, for some $m > n$ the construction  ensures that $\kappa_{(\dom{g_\sigma})}(m) > h(m)$.

Let $\{T_i\}_{i \in \N}$ be a descending sequence in $\po$ that meets all of the sets of the form \eqref{eq: dense} as well as those dense sets that ensure non-computability and being of computably dominated degree. If $X \in \bigcap_i[T_i]$, then $X$ is non-computable. Now if $Y\le_T X$ then because $X$ is of computably dominated degree $Y = \Phi^X$ for some truth-table functional $\Phi$. Hence for any computable function $h \in o(n)$, there is a partial computable function $g$ such that $Y$ extends~$g$ and $\kappa_{\dom(g)} \not \in o(h)$.
\end{proof}

One should notice that in the above proof, while the condition $\star$ is  $\Sigma^1_2$, the construction can be carried out using $\emptyset''$ because it is not necessary to determine if $\star$  holds. All that is required is to find a string $\sigma$ such that $f_\sigma$ is total.

\vspace*{1cm}

\textbf{Acknowledgements.} We would like to thank Alexander Shen for pointing out to us the notion of list decoding, which we only implicitly used in earlier presentations of our results. Thanks also go to Carl Jockusch and the anonymous referee for useful comments and suggestions.

      \bibliographystyle{plain}
      \bibliography{bib}

\end{document}